\documentclass[10pt]{article}

\usepackage[nottoc,notlot,notlof]{tocbibind}
\usepackage{amsmath,amssymb,amsfonts,amsthm}
\usepackage{latexsym}
\usepackage{tikz}
\usepackage[normalem]{ulem}
\usetikzlibrary{automata,positioning}
\usetikzlibrary{chains,fit,shapes}
\usetikzlibrary{calc}
\usetikzlibrary{arrows}
\usepackage{float}
\usetikzlibrary{positioning,calc}
\usetikzlibrary{graphs}
\usetikzlibrary{graphs.standard}
\usetikzlibrary{arrows,decorations.markings}
\usepackage{multicol}
\usepackage{xcolor}
\if@thmsec
  	\newtheorem{theorem}{Theorem}[section]
\else
   \newtheorem{theorem}{Theorem}
\fi

\newtheorem{lemma}[theorem]{Lemma}

\newtheorem{cor}[theorem]{Corollary}

  \newtheorem{dnt}[theorem]{Definition}

\numberwithin{equation}{section}
\title{On Patterns and Languages in $1$-$11$-Representations of Graphs}

\author{\hspace{1cm} Biswajit Das \ and Ramesh Hariharasubramanian \\ 
{{\footnotesize d.biswajit@iitg.ac.in},\ {\footnotesize  ramesh\_h@iitg.ac.in}}\\{\footnotesize Department of Mathematics, Indian Institute of Technology Guwahati, Guwahati, Assam 781039, India}}

\begin{document}
	\maketitle
	
	\begin{abstract}
A $1$-$11$-representation of a graph $G(V,E)$ is a word over the alphabet $V$ such that two distinct vertices $x$ and $y$ are adjacent if and only if the restricted word $w_{\{x,y\}}$ (obtained from $w$ by deleting all letters except $x$ and $y$) contains at most one occurrence of $xx$ or $yy$. Although every graph admits a $1$-$11$-representation, the repetition patterns that may or must appear in such representations have not been fully studied.

In this paper, we study cube-free and square-free $1$-$11$-representations of graphs. We first show that cubes cannot always be avoided in $1$-$11$-representations of minimum length by providing a graph for which every minimum-length $1$-$11$-representation necessarily contains a cube. We then focus on permutational $1$-$11$-representations, where the representing word is a concatenation of permutations of the vertex set. In this setting, we prove that any cube appearing in a permutational $1$-$11$-representation can be removed without changing the represented graph. As a consequence, every permutational $1$-$11$-representation attaining the permutational $1$-$11$-representation number is cube-free. We further show that this behaviour does not extend to squares by providing a graph for which every permutational $1$-$11$-representation with the minimum number of permutations necessarily contains a square.

Finally, we prove that the language of all $1$-$11$-representations of a given graph is regular. Moreover, we show that the language of all permutational $1$-$11$-representations of a graph is also regular.\\

	\end{abstract}
\section{Introduction}

An important problem at the interface of graph theory and theoretical computer
science is the efficient encoding of graphs. One approach is to represent a graph
by a word, where the alphabet consists of the vertices of the graph and adjacency
information is recovered from local patterns in the word. Many such encoding
schemes are possible, but those based on simple and uniform rules have proved
particularly useful.

A well-known example is the notion of word-representable graphs. In this notion,
two distinct vertices $x$ and $y$ are adjacent if and only if their corresponding
letters alternate in a word. This representation was first studied systematically
by Kitaev and Pyatkin \cite{Kitaev2008}, building on earlier ideas of Kitaev and Seif \cite{kitaev2008word}. Although this model captures a large class of graphs and has been widely investigated, it does not include all graphs.

 Another generalisation of word-representable graphs is given by $k$-$11$-representations, which were originally suggested by Remmel and later studied in detail by Cheon \textit{et al.} \cite{cheon2019k}. In this representation, an edge of a graph $G$ corresponds to at most $k$ occurrences of the consecutive pattern $11$ in a word representing $G$. As a result, word-representable graphs are exactly the $0$-$11$-representable graphs. It was shown in \cite{cheon2019k} that every graph is $2$-$11$-representable using a concatenation of permutations. Furthermore, $1$-$11$-representations were obtained for the Chv\'{a}tal graph, the Mycielski graph, split graphs, and graphs whose vertex sets can be partitioned into a comparability graph and an independent set \cite{Futorny2024}. More recently,
it was proved in \cite{hefty2025k} that every graph admits a permutational $1$-$11$-representation.

An important problem in the study of word-representable graphs is to determine representations that contain or avoid given patterns. A pattern $\tau = \tau_1\tau_2\cdots \tau_m$ is said to occur in a word
$w = w_1w_2\cdots w_n$ if there exist indices $1 \le i_1 < i_2 < \cdots < i_m \le n$ such that the word
$\tau_1\tau_2\cdots \tau_m$ is order-isomorphic to the subsequence $w_{i_1}w_{i_2}\cdots w_{i_m}$. The study of pattern-avoiding word-representations was introduced in \cite[Section~7.8]{kitaev2015words}. Later, Gao \textit{et al.} \cite{GKZ17-AJC} and Mandelshtam \cite{Mandelshtam19-DMGT} studied word-representable graphs avoiding the $132$-pattern and the $123$-pattern, respectively. Further results on forbidden-pattern characterisations can be found in \cite{takaoka2024forbidden,Damaschke90-incollection,FH21-SIDMA} and \cite[Section~7.4]{BLS99}.

Apart from permutation patterns, combinatorics on words also considers unordered patterns such as squares, cubes, overlaps, and borders. A word $u$ is said to contain a word $v$ as a factor if $u = xvy$, where the words $x$ and $y$ may be empty. A square (respectively, a cube) is obtained by repeating a nonempty word twice (respectively, three times) consecutively. Thus, a word
$w \in \Sigma^*$ contains a square (respectively, a cube) if it can be written in the form $w = s_1XXs_2$ (respectively, $w = s_1XXXs_2$), where $s_1, s_2 \in \Sigma^*$ and $X \in \Sigma^+$. Square-free words have attracted considerable attention in combinatorics on words since the work of Axel Thue \cite{thue1906uber}. This work initiated the study of combinatorics on words and established the existence of infinite square-free words over three-letter alphabets. An example of a word that avoids all non-trivial cubes is the Peano word \cite{Kitaev2004}. Square-free and cube-free word-representations of graphs were introduced in \cite[Section~7.1.3]{kitaev2015words}. It was shown there that every word-representable graph admits a cube-free word-representation, and that trivial square-free word-representations exist for all word-representable graphs except the empty graph on two vertices. It was later proved that non-trivial square-free word-representations exist for every word-representable graph except the empty graph on two vertices \cite{das2026square}.

Since not all graphs are word-representable, it is natural to ask whether non-word-representable graphs admit other representations that avoid repetitions, such as squares or cubes. From a formal-language perspective, it is known that for a word-representable graph the set of all words that represent the graph forms a regular language \cite{fleischmann2024word}. Since $1$-$11$-representations generalise word-representations, it is natural to ask whether the language of all $1$-$11$-representations of a given graph is also regular.

In this paper, we give an affirmative answer to this question for cubes. For word-representable graphs, it is shown in \cite{kitaev2015words} that cubes can always be removed from a word-representation. This property does not extend to $1$-$11$-representations. In particular, we construct a graph for which every minimum-length $1$-$11$-representation necessarily contains cubes, and such cubes cannot be removed without increasing the length of the representation.
We also show that this improves for permutational $1$-$11$-representations. Every graph admits a cube-free permutational $1$-$11$-representation, where the representing word is a concatenation of permutations of the vertex set. However, squares do not follow the same pattern. Although squares can be removed from any word-representation \cite{kitaev2015words,das2026square}, this is again not true for $1$-$11$-representations. The same graph used above also shows that square-free words cannot be obtained in minimum-length $1$-$11$-representations. For permutational $1$-$11$-representations, cube-free representations do exist. However, we give a graph for which every permutational $1$-$11$-representation with the minimum number of permutations necessarily contains a square. These results show that repetition avoidance is different for $1$-$11$-representations than for word-representations.

\section{Preliminaries}

All graphs considered in this paper are finite, simple, and undirected. For a graph $G(V,E)$, let $V^{+}$ denote the set of all nonempty words over the alphabet $V$.

 \begin{dnt}[\cite{kitaev2015words}, Definition 3.0.3.] Suppose that $w$ is a word and $x$ and $y$ are two distinct letters in $w$. The letters $x$ and $y$ alternate in $w$ if, after deleting all letters but the copies of $x$ and $y$ from $w$, either a word $xyxy\cdots$ (of even or odd length) or a word $yxyx\cdots$(of even or odd length) is obtained. If $x$ and $y$ do not alternate in $w$, then these letters are called non-alternating letters in $w$. 
 \end{dnt}

In a word $w$, if $x$ and $y$ alternate, then $w$ contains $xyxy\cdots$ or $yxyx\cdots$ (odd or even length) as a subword.

\begin{dnt}[\cite{kitaev2015words}, Definition 3.0.5.]
 A simple graph $G(V, E)$ is \textit{word-representable} if there exists a word $w$ over the alphabet $V$ such that the two distinct letters $x$ and $y$ alternate in $w$ if and only if $x$ and $y$ are adjacent in $G$. If a word $w$ \textit{represents} $G$, then $w$ contains each letter of $V(G)$ at least once.
\end{dnt}
For a word $w$, $w_{\{x_1, \cdots, x_m\}}$ denotes the word formed by removing all letters from $w$ except the letters $x_1, \ldots, x_m$.
\begin{dnt}[\cite{cheon2019k}]
A word $w\in V^{+}$ is called a \emph{$1$-$11$-representation} of a graph
$G(V,E)$ if, for every pair of distinct vertices $x,y\in V$, the vertices $x$
and $y$ are adjacent in $G$ if and only if the restricted word $w_{\{x,y\}}$
contains at most one factor of the form $xx$ or $yy$.

Equivalently, $x$ and $y$ are non-adjacent in $G$ if and only if the restricted
word $w_{\{x,y\}}$ contains at least two occurrences of $xx$, or at least two
occurrences of $yy$, or at least one occurrence of each.
\end{dnt}

\begin{dnt}[\cite{hefty2025k}]
   The \emph{$1$-$11$-representation number} of $G$, denoted by $\mathcal{R}(G)$, is
the minimum length of a word $w$ that $1$-$11$-represents $G$. 
\end{dnt}

\begin{dnt}[\cite{cheon2019k}]
A $1$-$11$-representation $w$ of $G$ is called \emph{permutational} if $w$ can be
written as a concatenation of permutations of $V$. In this case, we say that $G$
admits a permutationally  $1$-$11$-representation.    
\end{dnt}

\begin{dnt}[\cite{hefty2025k}]
$\mathcal{R}_\pi(G)$ denotes the permutational $1$-$11$-representation number of $G$, the minimum number of permutations in any permutational $1$-$11$-representation of $G$. 
\end{dnt}

\begin{theorem}[\cite{hefty2025k}, Theorem 2.2]
     Let $G(V, E)$ be a graph. Then there is a word $w$ over alphabet $V$
permutationally $1$-$11$-representing G.
\end{theorem}
In the following, we present the definitions of some patterns and the known results for these patterns.
 \begin{dnt}
    A \textit{square} (resp., \textit{cube}) in a word is two (resp., three) consecutive equal factors. A word $w\in \Sigma^*$ contains a square (resp., cube) if $w=s_1XXs_2$ (resp., $w=s_1XXXs_2$), where $s_1,s_2\in \Sigma^*,X\in \Sigma^+$.
\end{dnt}
\begin{theorem}[\cite{kitaev2015words}, Theorem 7.1.9.]\label{tmcube}
    For any word-representable graph $G$, there exists a cube-free word representing $G$.
\end{theorem}
\begin{theorem} [\cite{das2026square}, Theorem 2.] 
     If $G$ is a connected graph and $w$ is a word representing $G$ where $w$ contains at least one square, then there exists a square-free word $w'$ that represents $G$.
\end{theorem}
\begin{theorem}[\cite{das2026square}, Theorem 3.]\label{tm3}
Suppose $G$ is a disconnected word-representable graph, and $G_i$, $1\leq i\leq n$, $n\in \mathbb{N}$ are the connected components of $G$. Let $w_i$ be the square-free word-representation of $G_i$, and $G_1$ be a non-empty word-representable graph. Then the word $w=w_1\setminus l(w_1)w_2\cdots w_nl(w_1)\sigma(w_n)\cdots\sigma(w_2)$ $\sigma(w_1)\setminus l(w_1)\sigma(w_2)\cdots$ $\sigma(w_n)l(w_1)$ represents $G$ and $w$ is a square-free word. Here, $l(w_1)$ denotes the last letter of $w_1$ and $w_1\setminus l(w_1)$ denotes the word obtained by removing the last letter from $w_1$. 
\end{theorem}

\begin{lemma}[\cite{das2026square}, Lemma 4.]\label{lmk}
     If $G$ is a connected word-representable graph and the representation number of $G$ is $k$, then every $k$-uniform word representing $G$ is square-free.
\end{lemma}

\begin{cor}[\cite{das2026square}, Corollary 1]
     For a connected word-representable graph $G$, every minimal-length word representing $G$ is square-free.
\end{cor}
\begin{theorem}[\cite{fleischmann2024word},Theorem 26]
     The language of words representing a given graph is regular.
\end{theorem}
In this paper, the notation $w=a_1a_2\cdots a_n$ indicates that the word $w$ contains the factors $a_1,a_2,\ldots,a_n$, where each $a_i$ is a possibly empty word.

\section{Cube-free 1-11 representation of graphs}
In this section, we study cube-free $1$-$11$-representations of graphs.
While every graph admitting a $1$-$11$-representation may contain cubes in its
representing words, it is natural to ask whether such cubes are unavoidable or
can be eliminated without changing the represented graph. Our goal is to
understand when cube-free representations exist, and how cube structures behave
inside permutational $1$-$11$-representations.

We begin with a small disconnected graph to demonstrate that cube-free
$1$-$11$-representations need not exist at the minimum representation length.
This example shows that cube-freeness is a nontrivial restriction and motivates
the need for structural results in the permutation setting.

\begin{theorem}
Suppose that $G$ is a disconnected graph consisting of a component $K_3$ and an
isolated vertex $v$. Then $\mathcal{R}(G)=6$, and there does not exist any
cube-free word $w$ of length $6$ that is a $1$-$11$-representation of $G$.
\end{theorem}

\begin{proof}
Let $\{1,2,3\}$ be the vertex set of the component $K_3$. Since $v$ is isolated,
it is non-adjacent to each of $1,2,$ and $3$, while the vertices $1,2,3$ are
pairwise adjacent. Consider the set $S=\{123vvv,\ vvv123,\ 1vvv23,$ $\ 12vvv3\}$.
In each word in $S$, the letters $1,2,$ and $3$ alternate pairwise, ensuring
their mutual adjacency, while the factor $vvv$ guarantees non-adjacency between
$v$ and each of $1,2,$ and $3$. Hence, every word in $S$ is a valid
$1$-$11$-representation of $G$. Therefore, $\mathcal{R}(G)\le 6.$
We know that if $x$ and $y$ are non-adjacent vertices in a graph, then any
$1$-$11$-representation must contain either two occurrences of $xx$ or two
occurrences of $yy$, or at least one occurrence of each. Since $1,2,$ and $3$
are all non-adjacent to $v$, this condition must hold for each pair
$\{v,x\}$, where $x\in\{1,2,3\}$.

If $v$ occurs only twice in $w$ (that is, $vv$ occurs but $vvv$ does not), then
$vv$ can satisfy this condition for at most one pair $\{v,x\}$. Consequently, the remaining pairs force the occurrences of $11$, $22$, and $33$.
Thus, the word $w$ must have length at least $8$, which contradicts $|w|<6$. Hence, $\mathcal{R}(G)=6$.

Finally, every word in $S$ contains the factor $vvv$, which is a cube. Moreover,
if $|w|=6$ and $w$ is a $1$-$11$-representation of $G$, then $v$ must occur
exactly three times, and the remaining letters $1,2,$ and $3$ must occur exactly
once and alternate, forcing $w$ to belong to the set $S$. Hence, no cube-free
$1$-$11$-representation of $G$ of length $6$ exists.
\end{proof}

According to Theorem \ref{tmcube}, cubes can be removed from any word-representation of the same graph without changing the represented graph. However, this is not the case for $1$-$11$-representations. For word-representable graphs, two distinct vertices $x$ and $y$ are non-adjacent if and only if there is at least one occurrence of $xx$ or $yy$, so removing one occurrence from a cube preserves non-adjacency. In contrast, $1$-$11$-representations require at least two such occurrences, and removing a cube may destroy this condition. This example shows that even when the representation number is small, cubes may not be avoidable in minimal $1$-$11$-representations. In the remainder of this section, we therefore focus on permutation representations and investigate when cubes can be systematically removed.

We first show that cubes in permutational $1$-$11$-representations are highly
restricted. In particular, the length of the repeated block must be a multiple of the permutation length.
\begin{lemma}\label{lm1}
Suppose that $w$ is a permutational $1$-$11$-representation of a graph $G(V,E)$,
where $|V|=n$. Then $w$ cannot contain a cube $XXX$, where $X\in V^{+}$ and the
length of $X$ is not a multiple of $n$.
\end{lemma}

\begin{proof}
Suppose, for contradiction, that the word $w$ contains a factor $XXX$ such that
\[
|X| = i\times n + l,
\]
where $i\in \mathbb{Z}^{+}\cup\{0\}$ and $1\le l<n$. We consider two cases.

\medskip
\noindent\textbf{Case 1: $|X|<n$.}

Since $|X|<n$, not all vertices of $V$ can occur in $X$. As $X\in V^{+}$, there
exist distinct vertices $x,y\in V$ such that $x$ occurs in $X$ but $y$ does not.
Hence, the restricted word $(XXX)_{\{x,y\}}=xxx$. Thus, we obtain $w_{\{x,y\}}=s_1\,xxx\,s_2$, where $s_1,s_2\in\{x,y\}^{*}$. Since $w$ is a permutation representation, the three occurrences of $x$ belong to three distinct permutations $P_1$, $P_2$ and $P_3$ in $w$. Each permutation must contain $y$. The vertex $y$ may occur in $s_1$ and $s_2$, accounting for its presence in $P_1$ and $P_3$, but it cannot
occur in the permutation $P_2$, yielding a contradiction. Therefore, $|X|\ge n$.

\medskip
\noindent\textbf{Case 2: $|X| = i\times n + l$, where $i\in\mathbb{Z}^{+}$ and $1\le l<n$.}

Since $w$ is a permutation representation, the word $w$ is a concatenation of
permutations of $V$, each of length $n$. Consider the occurrence of the factor
$XXX$ in $w$. Because $|X|$ is not a multiple of $n$, each copy of $X$ in $w$
contains exactly $i$ complete permutations and a proper part of one additional
permutation. Suppose
\[
\begin{aligned}
w ={}& P_1P_2\cdots P_{i1}\,
P_{i2}P_{i+1}\cdots P_{j-1}P_{j1}\,
P_{j2}P_{j+1}\cdots P_{l-1}P_{l1} \\
& P_{l2}P_{l+1}\cdots P_{k-1}P_{k1}\,
P_{k2}\cdots P_p,
\end{aligned}
\]
where each $P_r$ and each concatenation $P_{r1}P_{r2}$ is a permutation of $V$.
Moreover,
\[
X =
P_{i2}P_{i+1}\cdots P_{j-1}P_{j1}
=
P_{j2}P_{j+1}\cdots P_{l-1}P_{l1}
=
P_{l2}P_{l+1}\cdots P_{k-1}P_{k1}.
\]

Since $l<n$, we have
$
1 \le |P_{i2}P_{j1}| = |P_{j2}P_{l1}| = |P_{l2}P_{k1}| < n.
$
Hence, there exist distinct vertices $x,y\in V$ such that $x$ occurs exactly $i$
times in $X$, while $y$ occurs $i-1$ or $i-2$ times in $X$, as each vertex must
occur once in every permutation of $w$. Consequently,
\[
|(XXX)_{\{x\}}| = 3\times i
\quad\text{and}\quad
3\times (i-2)\le |(XXX)_{\{y\}}|\le 3\times (i-1).
\]

Therefore, in the word $w$ there exist three permutations $P_a$, $P_b$, and
$P_c$ whose parts occurring in $XXX$ do not contain $y$. If two of these parts
belonged to the same copy of $X$, then that copy of $X$ would contain fewer
occurrences of $y$ than the other two, which is impossible. Hence, these parts
belong to three different copies of $X$.

Without loss of generality, assume that $P_a$ and $P_c$ contribute to the first
and third copies of $X$ in $XXX$. For these permutations, the unique occurrence
of $y$ lies outside $XXX$. The remaining permutation $P_b$ is either of the form
$P_{j1}P_{j2}$ or $P_{l1}P_{l2}$.

If $P_b=P_{j1}P_{j2}$, then $P_{j2}$ is the part of $P_b$ occurring in the second
copy of $X$ and does not contain $y$, so $y$ must occur in $P_{j1}$. Since $x$
occurs in $P_{i2}$, the first copy of $X$ contains the same number of occurrences
of $x$ and $y$, a contradiction. A similar contradiction arises when
$P_b=P_{l1}P_{l2}$, in which case $y$ must occur in $P_{l2}$, forcing equality of
the numbers of occurrences of $x$ and $y$ in the third copy of $X$.

Thus, $X$ cannot have length $i\times n+l$ with $1\le l<n$.
\end{proof}

Lemma \ref{lm1} shows that cubes whose length does not respect the permutation
block size cannot occur. This restriction plays a crucial role in identifying
which cubes are structurally removable.

Next, we show that consecutive identical permutation blocks can be removed
without changing the represented graph.
\begin{lemma}\label{lm2}
Suppose that $w$ is a permutational $1$-$11$-representation of a graph $G(V,E)$.
If $w$ contains a square $XX$ with $X=P_i$, then one occurrence of $P_i$ can be
removed without affecting the representation of $G$.
\end{lemma}

\begin{proof}
Suppose
\[
w = P_1P_2\cdots P_iP_i\cdots P_k,
\]
where each $P_j$ is a permutation of $V$. Let $w'$ be the word obtained from $w$
by deleting one occurrence of $P_i$. We show that $w'$ represents the same graph
$G$.

\medskip
\noindent\textbf{Case 1: Preservation of adjacency.}

Let $x,y\in V$ be adjacent in $G$. Since $w$ is a permutational $1$-$11$-representation,
the restricted word $w_{\{x,y\}}$ alternates, that is,
\[
w_{\{x,y\}} \in \{(xy)^t,\ (yx)^t,\ (xy)^{t_1}(yx)^{t_2},\ (yx)^{t_1}(xy)^{t_2}\}
\]
for some $t,t_1,t_2\ge1$.
In particular, the restriction of $P_i$ to $\{x,y\}$ is either $xy$ or $yx$.
Removing one copy of $P_i$ deletes exactly one occurrence of $xy$ or $yx$ from
$w_{\{x,y\}}$, and hence the remaining word $w'_{\{x,y\}}$ still alternates.
Therefore, $x$ and $y$ remain adjacent in $w'$.

\medskip
\noindent\textbf{Case 2: Preservation of non-adjacency.}

Let $x,y\in V$ be non-adjacent in $G$. Then $w_{\{x,y\}}$ contains at least one
occurrence of each $xx$ and $yy$, and can be written in one of the following
forms:
\[
(xy)^{k_1}(yx)^{k_2}(xy)s
\quad\text{or}\quad
(yx)^{k_1}(xy)^{k_2}(yx)s,
\]
where $k_1,k_2\ge1$ and $s\in\{x,y\}^*$.
Again, the restriction of $P_i$ to $\{x,y\}$ is either $xy$ or $yx$. Therefore, we obtain $(P_iP_i)_{\{x,y\}}=(xy)^2$ or $(yx)^2$. Deleting one copy of $P_i$ reduces one of the exponents $k_1$ or $k_2$ by $1$, but does not eliminate the occurrence of $xx$ or $yy$ in the restricted word. Hence, $w'_{\{x,y\}}$ still contains a square, and $x$ and $y$ remain non-adjacent in $w'$.

\medskip
In both cases, adjacency and non-adjacency are preserved. Therefore, removing
one occurrence of $P_i$ does not change the graph represented by $w$.
\end{proof}

Lemma \ref{lm2} shows that consecutive identical permutations can be removed
without changing the represented graph. This result is used to remove larger
repetitions such as cubes.

\begin{cor}\label{cor1}
      Suppose $w$ is a permutational $1$-$11$-representation of the graph $G(V,E)$. If $w$ contains a cube $XXX$, where $X=P_i$, then we can remove two $P_i$.
\end{cor}
\begin{proof}
    Directly derived from Lemma \ref{lm2}.
\end{proof}
We now combine the previous results to show that cubes in permutation
$1$-$11$-representations are not essential. In fact, any such cube can be
eliminated while preserving the represented graph.

\begin{theorem}\label{tm2}
If $w$ is a permutational $1$-$11$-representation of a graph $G(V,E)$ and $w$
contains a cube $XXX$, where $X\in V^{+}$, then there exists a cube-free
permutational $1$-$11$-representation of $G$.
\end{theorem}

\begin{proof}
Suppose that
\[
w = P_1P_2\cdots P_p,
\]
where each $P_i$ is a permutation of $V$, and assume that $w$ contains a cube
$XXX$ with $X\in V^{+}$.

If $X=P_i$ for some $i$, then by Lemma~\ref{lm2}, one may remove two occurrences
of $P_i$, and the resulting word still represents $G$. Hence, we may assume
that $X$ is not a single permutation. Thus, the factor $XXX$ occurs in $w$ in the following form:
\[
\begin{aligned}
w ={}& P_1P_2\cdots P_{i1}\,
P_{i2}P_{i+1}\cdots P_{j-1}P_{j1}\,
P_{j2}P_{j+1}\cdots P_{l-1}P_{l1} \\
& P_{l2}P_{l+1}\cdots P_{k-1}P_{k1}\,
P_{k2}\cdots P_p,
\end{aligned}
\]
where each $P_r$ and each concatenation $P_{r1}P_{r2}$ is a permutation of $V$,
and
\[
X =
P_{i2}P_{i+1}\cdots P_{j-1}P_{j1}
=
P_{j2}P_{j+1}\cdots P_{l-1}P_{l1}
=
P_{l2}P_{l+1}\cdots P_{k-1}P_{k1}.
\]

By Lemma~\ref{lm1}, the length of $X$ must be a multiple of $|V|=n$. Hence,
\[
|P_{i2}|+|P_{j1}|=|P_{j2}|+|P_{l1}|=|P_{l2}|+|P_{k1}|=n,
\]
and also
\[
|P_{i1}|+|P_{i2}|=|P_{j1}|+|P_{j2}|=|P_{l1}|+|P_{l2}|=|P_{k1}|+|P_{k2}|=n.
\]
Since each block $P_{ab}$, with $a\in\{i,j,l,k\}$ and $b\in\{1,2\}$, is a prefix or
suffix of $X$, it follows that $
P_{i2}=P_{j2}=P_{l2}
\quad\text{and}\quad
P_{j1}=P_{l1}=P_{k1}.
$
Moreover, the sequence of permutations between these boundary blocks must be the
same in all three copies of $X$. Consequently, the word $w$ can be rewritten as
\[
\begin{aligned}
w ={}& P_1P_2\cdots P_{i1}\,
X\,X\,X\,
P_{k2}\cdots P_p \\
={}& P_1P_2\cdots P_{i1}\,
\bigl(P_{i2}P_{i+1}\cdots P_{j-1}P_{j1}\bigr)^3\,
P_{k2}\cdots P_p .
\end{aligned}
\]

Let $w'$ be the word obtained from $w$ by removing the middle copy of $X$, that
is,
$
w' = P_1P_2\cdots P_{i1}\,
X\,X\,
$ $P_{k2}\cdots P_p .
$
We show that $w'$ is still a permutational $1$-$11$-representation of $G$.

\medskip
\noindent\textbf{Case 1: Preservation of adjacency.}

Let $x,y\in V$ be adjacent in $G$. Removing the middle copy of $X$ deletes the
same number of occurrences of $x$ and $y$. Since $P_{j1}P_{i2}$ is a permutation
of $V$, removing one occurrence of $P_{i2}$ and $P_{j1}$ does not create a factor
of the form $xx$ or $yy$. Hence, the alternating pattern of $x$ and $y$ in
$w_{\{x,y\}}$ is preserved in $w'_{\{x,y\}}$, and adjacency remains unchanged.

\medskip
\noindent\textbf{Case 2: Preservation of non-adjacency.}

Let $x,y\in V$ be non-adjacent in $G$. Then $w_{\{x,y\}}$ contains at least two
occurrences of either $xx$ or $yy$ or at least one occurrence of both. If at least one such occurrence lies entirely outside the cube $XXX$, then removing the middle copy of $X$ does not eliminate all such occurrences. If one occurrence lies within $X$, then another identical copy of $X$ remains in $w'$, preserving the non-adjacency. The only remaining possibility is that the two occurrences overlap the boundary between consecutive copies of $X$, which would require a repeated letter in $P_{j1}P_{i2}$. This is impossible since $P_{j1}P_{i2}$ is a permutation of $V$. Thus, non-adjacency is preserved in $w'$.

\medskip
Therefore, $w'$ represents the same graph $G$ as $w$.

Finally, this procedure can be applied iteratively to remove all cubes in $w$.
Since the middle copy of each cube is removed, no new cube can be created in the
process. Hence, we obtain a cube-free permutational $1$-$11$-representation of
$G$.
\end{proof}

Theorem \ref{tm2} shows that every permutational $1$-$11$-representation can be made
cube-free. Together with the earlier example, this shows that cube-free representations may not exist at minimum length, but they always exist for permutation representations.

The next result relates cube-freeness to the permutational $1$-$11$
representation number of a graph.

\begin{cor}
    Let $G(V,E)$ be a graph with $\mathcal{R}_\pi(G)=k$. Then every word $w$ consisting of $k$ permutations that represents $G$ is cube-free.
\end{cor}

\begin{proof}
    Suppose, $w=P_1P_2\cdots P_k$ is a word of $k$ permutations that $1$-$11$-represents the graph $G$ and that $w$ contains a cube. By Theorem~\ref{tm2}, this cube can be removed to obtain a shorter word $w'=P_1P_2\cdots P_\ell$ with $\ell<k$ that still $1$-$11$-represents $G$. This contradicts the assumption that $\mathcal{R}_\pi(G)=k$. Therefore, every word of $k$ permutations that represents $G$ must be cube-free.
\end{proof}
Thus, any permutational $1$-$11$-representation with the permutational
$1$-$11$-representation number is necessarily cube-free.

We now show that cube-free results for permutational $1$-$11$-representations do
not hold for square-free representations.

\begin{theorem}
Suppose that $G$ is a disconnected graph consisting of a component $K_3$ and an
isolated vertex $v$. Then $\mathcal{R}_\pi(G)=3$, and there does not exist any
square-free word consisting of two permutations that is a $1$-$11$-representation
of $G$.
\end{theorem}

\begin{proof}
Since only complete graphs admit a permutational $1$-$11$-representation with
one permutation, we have $\mathcal{R}_\pi(G)>1$. Let $\{1,2,3\}$ be the vertex
set of the component $K_3$. The word
$
w = 123\,v\,123\,v\,123
$
is a permutational $1$-$11$-representation of $G$. Hence, $
\mathcal{R}_\pi(G)\le 3.$

We now show that $\mathcal{R}_\pi(G)\neq 2$. Suppose, for contradiction, that
$w'=P_1P_2$ is a permutational $1$-$11$-representation of $G$. Recall that if
$x$ and $y$ are non-adjacent vertices, then any $1$-$11$-representation must
contain either two occurrences of $xx$, or two occurrences of $yy$, or at least
one occurrence of each. Since $v$ is non-adjacent to each of $1,2,$ and $3$, this
condition must hold for every pair $\{v,x\}$ with $x\in\{1,2,3\}$.

However, $w'$ cannot contain two occurrences of $vv$ or $xx$, since each
permutation contains each vertex exactly once. Thus, for each $x\in\{1,2,3\}$,
the restricted word $w'_{\{v,x\}}$ must contain exactly one occurrence of $vv$
and one occurrence of $xx$. Consequently, $
w'_{\{v,1\}} \in \{vv11,\,11vv\}.
$
In either case, one of the permutations $P_1$ or $P_2$ would have to contain two
occurrences of $v$ or two occurrences of $1$, which is impossible. Hence, $w'$
cannot represent $G$, and therefore $\mathcal{R}_\pi(G)=3$.

Next, suppose that there exists a square-free word
$
w = P_1P_2P_3,
$
where each $P_i$ is a permutation of $V$, that $1$-$11$-represents $G$. Since
each vertex appears exactly once in each permutation, every vertex appears
exactly three times in $w$. As $w$ is square-free, it cannot contain a factor of
the form $vvv$ or $111$, $222$, or $333$.

For each $x\in\{1,2,3\}$, the restricted word $w_{\{v,x\}}$ must therefore be of
the form $xvvxxv$ or $vxxvvx$. If $w$ begins with $v$, then
$
w_{\{v,1\}} = v11vv1,\quad
w_{\{v,2\}} = v22vv2,\quad
w_{\{v,3\}} = v33vv3.
$
In this case, none of the vertices $1,2,$ or $3$ can occur between the second and
third occurrences of $v$, which produces a square $vv$ in $w$. This
contradicts the assumption that $w$ is square-free. Thus, $w$ must begin with one of $1,2,$ or $3$. Suppose
$
w_{\{v,1\}} = 1vv11v,\quad
w_{\{v,2\}} = 2vv22v,\quad
w_{\{v,3\}} = 3vv33v.
$
Then the first occurrences of $1,2,$ and $3$ occur before the first occurrence of
$v$, and their second occurrences occur after the second occurrence of $v$,
which again forces a square $vv$ in $w$.

To avoid this, at least one vertex from $\{1,2,3\}$ must occur twice between the
first and second occurrences of $v$, and another must occur twice between the
second and third occurrences of $v$. Without loss of generality, suppose $2$
occurs twice between the first and second occurrences of $v$, and $3$ occurs
twice between the second and third occurrences of $v$. Since $2$ and $3$ are adjacent, the restricted word $w_{\{2,3\}}$ may contain at most one of the factors $22$ or $33$, but not both. However, from $
w_{\{v,2\}} = v22vv2 \quad \text{and} \quad w_{\{v,3\}} = 3vv33v,
$
neither $2$ nor $3$ can occur between $33$ or $22$, respectively. This implies that $w_{\{2,3\}}$ contains both $22$ and $33$, which contradicts the adjacency of $2$ and $3$.

Hence, any permutational $1$-$11$-representation of $G$ with three permutations must contain a square. Therefore, $G$ does not admit a square-free permutation $1$-$11$-representation.
\end{proof}
Together with the previous corollary, this shows that while cubes can always be eliminated at the permutational $1$-$11$-representation number, squares may be unavoidable. Since these conditions depend only on fixed patterns in the word, we now study the set of all such representations from a formal language perspective.

\section{The language of $1$-$11$-representations}

In this section, we show that the language of all $1$-$11$-representations of a given graph $G$ is regular by describing a regular language that captures such representations. We begin by recalling the notion of a deterministic finite automaton.

A \emph{deterministic finite automaton} (DFA) is a $5$-tuple $\mathcal{A}=(Q,V,\delta,q_0,F)$, where $Q$ is a finite set of states, $V$ is a finite input alphabet, $\delta:Q\times V\to Q$ is the transition function, $q_0\in Q$ is the initial state, and $F\subseteq Q$ is the set of accepting states.

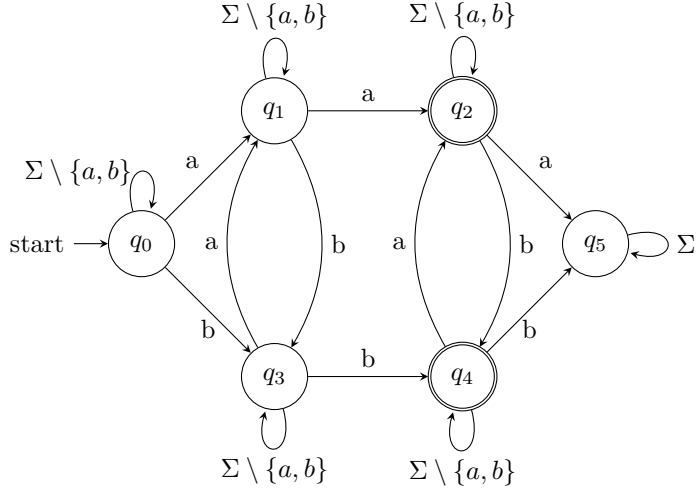
\begin{figure}[h]
    \centering
\begin{tikzpicture}[->, >=stealth, node distance=2.5cm, on grid, auto]

\node[state, initial] (q0) {$q_0$};
\node[state, above right=of q0] (q1) {$q_1$};
\node[state, accepting, right=of q1] (q2) {$q_2$};
\node[state, below right=of q0] (q3) {$q_3$};
\node[state, accepting, right=of q3] (q4) {$q_4$};
\node[state, below right=of q2] (q5) {$q_5$};

\path
(q0) edge node {a} (q1)
     edge node [below] {b} (q3)
     edge[loop above] node [left] {$\Sigma\setminus \{a,b\}$} ()
(q1) edge node {a} (q2)
     edge[bend left] node {b} (q3)
     edge[loop above] node {$\Sigma\setminus \{a,b\}$} ()
(q2) edge node {a} (q5)
     edge[bend left] node {b} (q4)
     edge[loop above] node {$\Sigma\setminus \{a,b\}$} ()
(q3) edge[bend left] node {a} (q1)
     edge node {b} (q4)
     edge[loop below] node {$\Sigma\setminus \{a,b\}$} ()
(q4) edge[bend left] node {a} (q2)
     edge node[below] {b} (q5)
     edge[loop below] node {$\Sigma\setminus \{a,b\}$} ()
(q5) edge[loop right] node {$\Sigma$} ();

\end{tikzpicture}
\caption{A DFA recognizing the language $L_{a,b}$.}
\label{fig:dfa}
\end{figure}
The DFA in Figure~\ref{fig:dfa} recognises the language
\[
L_{a,b}
=
\left\{
w \in \Sigma^*
\;\middle|\;
\begin{array}{l}
a,b \in \Sigma,
w_{\{a,b\}} \text{ contains at least one occurrence of each of } \\ a \text{ and } b,
w_{\{a,b\}} \text{ contains at most one occurrence of } aa \text{ or } bb,\\
\text{symbols from } \Sigma \setminus \{a,b\} \text{ may appear anywhere in } w
\end{array}
\right\}.
\]

This language captures the adjacency condition in a $1$-$11$-representation of a graph. Two vertices $a$ and $b$ are adjacent if and only if $w_{\{a,b\}}$ contains at most one occurrence of $aa$ or $bb$; they are non-adjacent otherwise. Since regular languages are closed under complement, the language $\overline{L}_{a,b}$ capturing non-adjacency is also regular.
\newpage
\begin{theorem}\label{thm1}
The language of all $1$-$11$-representations of a given graph $G(V,E)$ is
regular.
\end{theorem}

\begin{proof}
Let the alphabet be $V$. For any pair of distinct vertices $a,b\in V$, let $L_{a,b}$ denote the set of all words $w\in V^*$ such that the subword $w_{\{a,b\}}$ contains at least one occurrence of each of $a$ and $b$, and at most one occurrence of $aa$ or $bb$. Let $\overline{L}_{a,b}$ denote the complement of $L_{a,b}$ in $V^*$.

Let $L$ denote the language of all $1$-$11$-representations of the graph $G$. A word $w\in V^*$ belongs to $L$ if and only if $w$ belongs to $L_{a,b}$ for every pair of vertices $a$ and $b$ that are adjacent, and belongs to $\overline{L}_{a,b}$ for every pair of vertices $a$ and $b$ that are non-adjacent.

Equivalently, $L$ is the intersection of all languages $L_{a,b}$ taken over adjacent pairs of vertices $a$ and $b$, and all languages $\overline{L}_{a,b}$ taken over non-adjacent pairs of vertices $a$ and $b$. For each pair of vertices $a,b\in V$, the language $L_{a,b}$ is regular, as recognised by the DFA in Figure~\ref{fig:dfa}. Since regular languages are closed under complement and finite intersection, the language $L$ is regular.
\end{proof}

\begin{cor}
The language of all permutational $1$-$11$-representations of a given graph $G(V,E)$ is regular.
\end{cor}

\begin{proof}
Let $P$ denote the set of all permutations of $V$. Since $V$ is finite, $P$ is a finite language and therefore regular. The language $P^+$, consisting of all finite concatenations of permutations of $V$, is also regular.

Let $L'$ denote the language of all $1$-$11$-representations of $G$. A word is a permutational $1$-$11$-representation if and only if it belongs to both $P^+$ and $L'$. Since regular languages are closed under concatenation and intersection, the language of all permutational $1$-$11$-representations of $G$ is regular.
\end{proof}

\section{Conclusion}
In this paper, we studied repetition patterns in $1$-$11$-representations of graphs and showed that cube-free representations can always be obtained in the permutation framework, while square-free representations may be unavoidable. We also established that the language of all $1$-$11$-representations of a given graph, as well as the language of all permutational $1$-$11$-representations, is regular, highlighting a strong connection between graph representations and formal language theory.

\end{document}